\documentclass[a4paper,12pt]{article}

\usepackage[latin1]{inputenc}
\usepackage[T1]{fontenc}
\usepackage[english]{babel}

\usepackage{mathrsfs}
\usepackage{amscd}
\usepackage{amsfonts}
\usepackage{amsmath}
\usepackage{amssymb}
\usepackage{amstext}
\usepackage{amsthm}
\usepackage{amsbsy}

\usepackage{xspace}
\usepackage[all]{xy}
\usepackage{graphicx}
\usepackage{url}
\usepackage{latexsym}

\makeatletter
\newcommand*{\rom}[1]{\expandafter\@slowromancap\romannumeral #1@}
\makeatother

\theoremstyle{definition}

\newtheorem{fact}{fact}

\newtheorem{thm}[fact]{Theorem}
\newtheorem{lemma}[fact]{Lemma}
\newtheorem{prop}[fact]{Proposition}
\newtheorem{corollary}[fact]{Corollary}
\newtheorem{defini}[fact]{Definition}

\begin{document}

\title{Some Observations on Infinitary Complexity}
\author{Merlin Carl}

\maketitle

\begin{abstract}
 Continuing the study of complexity theory of Koepke's Ordinal Turing Machines (OTMs) that was done in \cite{CLR}, we prove the following results:

\begin{enumerate}
 \item An analogue of Ladner's theorem for OTMs holds: That is, there are languages $\mathcal{L}$ which are NP$^{\infty}$, but neither P$^{\infty}$ nor NP$^{\infty}$-complete. This answers an open question of \cite{CLR}.
 \item The speedup theorem for Turing machines, which allows us to bring down the computation time and space usage of a Turing machine program down by an aribtrary positive factor under relatively mild side conditions by expanding the working alphabet 
  does not hold for OTMs.
 \item We show that, for $\alpha<\beta$ such that $\alpha$ is the halting time of some OTM-program, there are decision problems that are OTM-decidable in time bounded by $|w|^{\beta}\cdot\gamma$ for 
some $\gamma\in\text{On}$, but not in time bounded by $|w|^{\alpha}\cdot\gamma$ for any $\gamma\in\text{On}$.
\end{enumerate}

\end{abstract}

\section{Introduction}

After the introduction of Infinite Time Turing Machines (ITTMs) in \cite{HL} and the subsequent development of various other infinitary machine models of computation e.g. in \cite{wITRM}, \cite{KS}, \cite{ITRM}, \cite{OTM}, \cite{ORM},
analogues of several central topics in classical computability theory were developed for these machine types, among them degree theory \cite{W1}, computable model theory \cite{CH}, randomness (\cite{CS}, \cite{C14},\cite{CS2}) and complexity theory.
Complexity theory was first studied by Schindler in the case of ITTMs, who proved that $\text{P}\neq\text{NP}$ for ITTMs \cite{Schindler}, which was later refined in various ways \cite{DHS}, \cite{HW}. 
Results on the space complexity for infinitary computations were given by Winter in \cite{Wi1}, \cite{Wi2}
It was occasionally remarked that
complexity theory for ITTMs is somewhat unsatisfying due to the fact that all inputs for ITTMs have the same length, namely $\omega$. 

This motivated the consideration of complexity theory for `symmetrical' models that have the same amount of time and space available, the most prominent of which are Koepke's `Ordinal Turing Machines' (OTMs), which can be thought 
of as Turing machines with a tape of proper class length $\text{On}$ and unbounded ordinal same working time. For an introduction to OTMs, we refer to \cite{OTM}. In agreement with the theory of classical Turing machines,
we explicitely allow multitape-OTMs, i.e. OTMs with any finite number of scratch tape. The study of complexity theory for these machines was started
by L\"owe in \cite{L}. After this, the subject lay dormant for a while, until it was revived by L\"owe and Rin at the CiE 2016, which led to \cite{CLR}. 
The central contributions of that paper were the introduction
of natural infinitary analogues of the classes P and NP, called P$^{\infty}$ and NP$^{\infty}$ and of the satisfaction problem SAT for OTMs, called SAT$^{\infty}$, the proof of a corresponding Cook-Levin theorem 
showing that SAT$^{\infty}$ is NP$^{\infty}$-complete, 
and the proof that SAT$^{\infty}$ (and hence any other NP$^{\infty}$-complete problem) is in fact not OTM-computable. 

Among the questions left open in \cite{CLR} was whether there is an analogue of Ladner's theorem for OTMs, i.e. whether there are problems in NP$^{\infty}\setminus$P$^{\infty}$ that are not NP$^{\infty}$-complete. 


In this paper, we response to this question by showing that the OTM-analogue of Ladner's theorem holds. We then use the same proof idea to show that the hierarchy of OTM-decision problems decidable with time bound $|w|^{\alpha}\cdot\gamma$ for some
ordinal $\gamma$ is strictly increasing in $\alpha$ (where $\alpha$ is the halting time of some OTM-program) and that there is 
no analogue of the speedup-theorem for Ordinal Turing Machines.



\section{Preliminaries}

We start by explaining the notations and giving the results that will be used in the course of this paper. Those which are not folklore can be found \cite{CLR}. For the definition of the complexity classes P$^{\infty}$, NP$^{\infty}$ as well
as the problem SAT$^{\infty}$, we also refer to \cite{CLR}.

We say that a structure $(S,E)$, $E\subseteq S\times S$, is coded by $\beta\in\text{On}$ if and only if there is some bijection
$f:\gamma\rightarrow S$, $\gamma\in\text{On}$ and $c:=\{p(\iota_{1},\iota_{2}):f(\iota_{1})Ef(\iota_{2})\}$, where $p$ is Cantor's pairing function.

$\{0,1\}^{**}$ is the set of functions mapping some ordinal to $\{0,1\}$.
For $x\in\{0,1\}^{**}$, $|x|$ denotes the length of $x$, i.e. its pre-image.

For an ordinal $\beta$, denote by $\beta_{0}$ and $\beta_{1}$ the summand $<\omega$ and the rest, respectively, when $\beta$ is written 
in Cantor normal form, i.e. $\beta=\omega\beta_{1}+\beta_{0}$, $\beta_{0}<\omega$. In this way, every ordinal naturally corresponds to a pair consisting of a multiple of $\omega$ and a natural number.

Let $(P_{i}:i\in\omega)$ enumerate the OTM-programs in some natural way.

The following is the main result of \cite{CLR}:

\begin{thm}{\label{OTM SAT}}
The satisfaction problem for infinitary propositional formulas (conjunctions and disjunctions of any ordinal length are allowed) is NP$^{\infty}$-complete. At the same time, it is OTM-undecidable.
\end{thm}
\begin{proof}
See \cite{CLR}, Theorem $4$, Theorem $5$ and Theorem $10$.
\end{proof}

\section{An Analogue of Ladner's Theorem}


Ladner's theorem answers the question whether any NP-problem that is not in P is already NP-complete. Of course, 
this is trivially true if it should happen that P=NP. Thus, Ladner's theorem is stated in a conditional form \cite{FG}: 
If P$\neq$NP, there is
$A\in\text{NP}\setminus\text{P}$ such that $A$ is not NP-complete. We will now show that an analogous result holds for OTMs. 
Since we know that P$^{\infty}\neq$NP$^{\infty}$ from Theorem \ref{OTM SAT}, we can state it unconditionally. Also by Theorem \ref{OTM SAT}, 
it will suffice to find a decidable problem in NP$^{\infty}\setminus$P$^{\infty}$ to prove this. Such a problem will now be constructed by a diagonalization. 
Although considerably different in the details, the proof is `morally´an adaption of the second one given in \cite{FG}.


Such a problem will now be constructed by a diagonalization. We start with some preliminary results that will be helpful in the construction.

\begin{lemma}{\label{wf part admissible}}
If $M\models\text{ZFC}^{-}$, then the well-founded part of $M$ is admissible and thus closed under ordinal exponentiation.
\end{lemma}
\begin{proof}
See \cite{H}, Lemma 5.1.
\end{proof}

\begin{prop}{\label{codes}}
For each infinite ordinal $\alpha$, there is a subset of $\alpha^{2}$ that codes a model of $\text{ZFC}^{-}$ that has $\alpha$ in its well-founded part and moreover codes $\iota\in\alpha$ by $\iota$.
\end{prop}
\begin{proof}
This can be achieved by first observing that there must be such a model of the right cardinality (form the elementary hull of $\alpha+1$ in some $L_{\beta}\models\text{ZFC}^{-}$ with $\beta>\alpha$, then use condensation) 
and then re-organizing the code, if necessary.
\end{proof}

\begin{lemma}{\label{quick checking}}
Checking whether some $x\subseteq\alpha$ codes a model of a certain first-order sentence $\phi$ is possible in polynomial time in $\alpha$, in fact in time $\alpha^{\omega}$ (in fact in time $\alpha^{n}$ when $\phi$ contains $n$ quantifiers and occurences of $\in$). 
 Hence, checking whether a subset of an ordinal $\alpha$ codes a model of $\text{ZFC}^{-}$ is possible in time $\alpha^{\omega}\omega$, which is still polynomial in $\alpha$.

\end{lemma}
\begin{proof}
This is done by exhaustively searching through the code for every quantifier and evaluating the logical connectives in the obvious way (technically, evaluating the relation needs another 
searching through the code, which is the reason for the exponent mentioned above).

We refer to \cite{OTM}, Lemma 6.1 for a more detailed description of the algorithm.

\end{proof}

\begin{thm}{\label{inf ladner}}
There is a subclass $X$ of $\{0,1\}^{**}$ which is NP$^{\infty}$, but neither P$^{\infty}$ nor NP$^{\infty}$-complete. In fact,
$X\in$NP$^{\infty}\setminus$P$^{\infty}$ can be chosen to be OTM-decidable.
\end{thm}

\begin{proof}
We construct such a problem $X\subseteq\{0,1\}^{**}$ by diagonalization. 
Let $X:=\{x\in\{0,1\}^{**}:P_{|x|_{0}}(x)\text{ does not halt in }\leq|x|_{1}^{|x|_{1}}|x|_{1}$ 



$\text{ many steps or does halt
in that many steps but rejects (i.e. outputs }0)\}$.

It is easy to see that $X$ is OTM-decidable: Given $x$, simply simulate $P_{|x|_{0}}(x)$ for $|x|_{1}^{|x|_{1}}|x|_{1}$ many steps and then flip the output (i.e. accept if the simulated computation rejects or does not halt, otherwise reject).

It is also clear that $X$ is not in $P^{\infty}$: If $P_{k}$ was an OTM-program that decides $X$ in time $\leq|x|^{\alpha}\beta$,
let $x\in\{0,1\}^{**}$ be such that $|x|>\text{max}\{\alpha,\beta\}$ and $|x|_{0}=k$; then $P_{k}(x)$ will give the wrong result by definition of $X$.

It remains to see that $X$ is in NP$^{\infty}$. Consider the class 

$X^{\prime}:=\{(x,y):x,y\in\{0,1\}^{**}\wedge`y\subseteq|x|^{2}\text{ codes a ZFC}^{-}\text{-model }M$ 

$\text{ with well-founded part of height }>|x|\text{'}\wedge`M\text{ believes that }P_{|x|_{0}}(x)\text{ does not halt in }\leq |x|_{1}^{|x|_{1}}|x|_{1}\text{ many steps or does halt in }
\leq |x|_{1}^{|x|_{1}}|x|_{1}\text{ many steps but rejects (i.e. outputs }0)\text{'}\}$. Such a code exists by Proposition \ref{codes} above.

The statement just given is a first-order statement in the parameter $x$ and can be evaluated in time polynomial in $|x|+|y|$, which is polynomial in $|x|$ as $|y|\leq|x|^{2}$ by assumption. 
By Lemma \ref{wf part admissible}, the computation within $M$ will belong to the well-founded part of $M$, and thus be an actual computation in $V$, so that $M$ will be correct about the result.

Hence $X^{\prime}$ belongs to $P^{\infty}$, and $X$, as the projection of $X^{\prime}$ to the first component, belongs to NP$^{\infty}$.

Thus $X$ is indeed an OTM-decidable (and thus NP$^{\infty}$-incomplete) problem in NP$^{\infty}\setminus$P$^{\infty}$, 
so $X$ is as desired.
\end{proof}

The above proof, only depending on the fact that KP-models are closed under ordinal polynomials, actually shows much more than P$^{\infty}\neq$NP$^{\infty}$. In fact, NP$^{\infty}$ is an extremely rich class:
For example, let us say that a class $X\subseteq\{0,1\}^{**}$ is EXPTIME$^{\infty}$ if and only if there are an OTM-program $P$, an ordinal $\alpha$ and an ordinal polynomial $p$ such that $P$ decides $X$
and works for $\leq\alpha^{p(\beta)}$ many steps on an input of length $\beta$. Similarly, let $X$ be EXPEXPTIME$^{\infty}$ if and only if this works with time bound $\alpha^{\alpha^{p(\beta)}}$. Then, by the argument, one also obtains:

\begin{corollary}
EXPTIME$^{\infty}$ and EXPEXPTIME$^{\infty}$ are properly contained in NP$^{\infty}$.
\end{corollary}

\section{Speedup and Strictness of the transfinite Polynomial Hierarchy}

A well-known theorem from classical complexity theory is the speedup-theorem, see e.g. [Hro]. This theorem says that, under certain mild conditions about the function $f$, if $0<c<1$ and there 
is a Turing program for deciding a certain language within time or space bounded by $f$ in the length of the input, then there is another Turing program deciding this language in time or space bounded
by $cf$ in the length of the input. This observation is crucial for classical complexity theory, as it justifies the introduction of $O(f)$-classes for measuring complexities. 
The proof idea is to let the new machine work on a considerably enriched alphabet, in which long strings of symbols of the original alphabet are condensed into
one symbol and thus processed in much fewer steps.

It is rather obvious that this approach will not work for infinitary machines: First, the alphabet is restricted to $\{0,1\}$; however, this is a formal limitation that could be overcome by slight changes
in the definition. More importantly, such a compression of the alphabet will not have much of an effect, since a finite time compression will not reduce the working time when it is a limit ordinal.

In fact, there are speedup theorems also for infinitary machines, such as the speedup-theorem for Infinite Time Turing Machines
by Hamkins and Lewis. These appear in the context of clockable ordinals, but they give in a sense only a much weaker speedup: Namely,
if there is a program $P$ that halts in $\alpha+n$ many steps and $1<n\in\omega$, then there is a program $P^{\prime}$ that halts
in $\alpha+1$ many steps. Note that this statement makes no reference to decision problems.

We will now show that there is in fact no analogue of the classical speedup-theorem for OTMs by showing that
there is a decision problem that is solvable in running time $\alpha\cdot 4$, but not in running time $\alpha\cdot 2$.

\begin{thm}{\label{noOTMspeedup}}
 There is a decision problem $\mathcal{L}\subseteq\Sigma^{**}$ such that $\mathcal{L}$ is decidable
in running time $\alpha\cdot 4$, but not in running time $\alpha\cdot 2$ (where $\alpha$ denotes the length of the input).
\end{thm}
\begin{proof}
We prove this by diagonalization. To this end, we consider the language $\mathcal{L}\subseteq\{0,1\}^{**}$, where $w\in\{0,1\}^{**}$
belongs to $\mathcal{L}$ if and only if the following holds: Let $w^{\prime}$ be the initial segment
of $w$ of length $\omega$. Let $i=0$ if $w^{\prime}$ consists entirely of $1$s, and let $i$ be the length of the longest
initial segment of $w^{\prime}$ consisting entirely of $1$s otherwise. Now run $P_{i}$ on input $w$
for $|w|\cdot 2$ many steps. If the output is $1$, output $0$, otherwise (i.e. if the output is different from $1$
or there is no output as the program didn't halt in that time) output $1$.

\bigskip
\noindent
\textbf{Claim 1}: $\mathcal{L}$ is not OTM-decidable in time complexity $\alpha\cdot 2$. 

\bigskip
For suppose that $P_{j}$ was a program that 
decides $\mathcal{L}$ in running time bounded by $|w|\cdot 2$. Consider a word $w$
of the form $w=\underbrace{11...1}_{j\times}0w^{\prime}$ with $|w^{\prime}|>\omega$. If $P_{j}$ with input $w$ does not halt in $<|w|\cdot 2$ many steps,
it does not decide $\mathcal{L}$ in the desired running time. Otherwise, its output will be wrong by definition of $\mathcal{L}$.

\bigskip
\noindent
\textbf{Claim 2}: $\mathcal{L}$ is OTM-decidable in time complexity $\alpha\cdot 4$. 

\bigskip
To see this, we use a multitape-OTM, i.e. an OTM with multiple (but finitely many) tapes. 
Given $w$, we run through the first $\omega$ many symbols to determine $i$. Then, we write the $i$th OTM-program
to an extra tape, which will later on direct the simulation of $P_{i}$ on input $w$.
Further, we run through $w$ from left to right, marking an extra field with $1$ on two extra tapes $T_{0},T_{1}$ for each 
symbol of $w$; these will serve as a `stopwatch' for our simulation.
Now simulate $P_{i}$ on $w$; each simulation step will only take a bounded finite number $c$ of computation steps, which depends
only on $i$. For each simulation step, move the head to the right first on $T_{0}$ and, when the head arrives at the right border of $T_{0}$,
continue on $T_{1}$; when the right border of $T_{1}$ has been reached, stop the simulation.

The first phase needs $\omega+\alpha$ many steps, which is $\alpha$ for $\alpha$ sufficiently large (i.e. $\alpha\geq\omega^{2}$).
The simulation then takes $c\alpha\cdot 2$ many steps. Writing $\alpha=\omega\alpha^{\prime}+k$, $k\in\omega$,
we have $c(\omega\alpha^{\prime}+k)\cdot 2=(c\omega\alpha^{\prime}+ck)\cdot 2=(\omega\alpha^{\prime}+ck)\cdot 2=\omega\alpha^{\prime}\cdot2+2ck<
\alpha\cdot 4$, so that we get $<\alpha\cdot 4$ many steps in total.
\end{proof}

\noindent
\textbf{Remark}: As one can easily see from inspecting the proof, neither the choice of the constants nor of the 
function $\alpha\mapsto\alpha$ instead of e.g. $\alpha\mapsto\alpha^{2}$ makes a difference.

\bigskip
We now define a rather natural hierarchy on the $\infty$-polynomially decidable decision problems.

\begin{defini}
 For $\alpha\in\text{On}$, let us say that a class $X\subseteq\{0,1\}^{**}$ is $\mathcal{P}_{\alpha}$ if and only if there is an OTM-program $Q$
and $\beta\in\text{On}$ such that $Q$ decides $X$ and takes less than $\gamma^{\alpha}\beta$ many steps on an input of length $\gamma$.
\end{defini}

Clearly, every $\mathcal{P}_{\alpha}$-class is also $\mathcal{P}_{\beta}$ for $\alpha\leq\beta$. The $\mathcal{P}_{\alpha}$-classification is thus a stratification of the class P$^{\infty}$, which we
name the $\infty$-polytime hierarchy.

An easy adaption of the argument used for Theorem \ref{noOTMspeedup} yields:

\begin{corollary}
 The $\infty$-polytime hierarchy is strict: If $\alpha<\beta$ and $\alpha$ is the halting time of some OTM-program, there is a class $X\subseteq\{0,1\}^{**}$ which is $\mathcal{P}_{\beta}$, but not $\mathcal{P}_{\alpha}$.
\end{corollary}
\begin{proof}
Given $w$, define $i=i(w)$ as in the proof of Theorem \ref{noOTMspeedup}. Now run $P_{i}$ on input $w$ for $|w|^{\alpha}|w|$ many steps. If $P_{i}$ halts in that many steps with output $1$, then let $w\in\mathcal{L}$, otherwise $w\notin\mathcal{L}$. 

By the usual argument, $\mathcal{L}$ is not decidable in time bounded by the function $x^{\alpha}\cdot\gamma$ for any $\gamma\in\text{On}$: 
To see this, just assume that $P_{j}$ is an OTM-program that decides $\mathcal{L}$ within that time bound and pick $w\in\{0,1\}^{**}$ such that $i(w)=j$ and $|w|>\gamma$.

To see that $\mathcal{L}$ belongs to $\mathcal{P}_{\beta}$, notice that, for $|w|>\alpha$, we have
 $|w|^{\beta}\cdot 2>|w|^{\alpha}|w|$. Thus, $|w|^{\beta}\cdot\omega$ steps suffice to simulate $P_{i(w)}$ for
$|w|^{\alpha}|w| $ many steps on input $w$ and flip the output, which decides $\mathcal{L}$. As $\alpha$ is by assumption the halting time of some OTM-program, the time bound $|w|^{\alpha}|w|$ is OTM-computable in the input $w$.

\end{proof}



\section{Conclusion and Further Work}

In many respects, the complexity theory of OTMs resembles classical complexity theory; typical results from classical complexity theory also hold in the OTM-concept and can often be proved
by adaptions of the classical arguments to the infinitary framework. This suggests studying infinitary analogues of decision problems considered in classical complexity theory
and to see what their OTM-complexity is. In \cite{CLR}, this was done for SAT, and it turned out that the usual proof of the Cook-Levin theorem could be adapted to yield an analogue for OTMs.
However, things do not always go that smoothly: For example, while the independent set problem  (i.e. determining whether a given graph $G$
has a subset of $n$ vertices, no two of which are connected) has a straightforward infinitary analogue (which is obtained by replacing $n$ with an arbitrary ordinal), 
the classical reduction of SAT to the independent set problem no longer works in the infinitary context since infinitary sets can have infinitary subsets of the same size.


This motivates the general, if somewhat vague question: What is it about decision problems in the classical sense that allows an infinitary generalization? Is there a general transfer principle?

On the other hand, we plan to explore the complexity classes of problems from infinitary combinatorics, such as the existence of infinite paths in a given tree.


\end{document}